\documentclass[a4paper,11pt]{article}
\usepackage[top=3.0cm, bottom=3.0cm, inner=3.0cm, outer=3.0cm, includefoot]{geometry}
\usepackage{amsmath}
\usepackage{amsthm}
\usepackage{amsfonts}
\usepackage{tcolorbox}
\usepackage{graphicx}
\usepackage{authblk}
\usepackage{hyperref}
\hypersetup{hidelinks}

\usepackage[numbers,sort&compress]{natbib}

\title{Monotonic normalized heat diffusion for distance-regular graphs with classical parameters of diameter $3$}

\author[1]{Shiping Liu\thanks{Email: spliu@ustc.edu.cn}}
\author[2]{Heng Zhang\thanks{Corresponding author. Email: hengz@mail.ustc.edu.cn}}

\affil[1,2]{School of Mathematical Sciences, University of Science and Technology of China, Hefei 230026, China}

\date{} 

\begin{document}

\maketitle

\newtheorem{thm}{Theorem}
\newtheorem{prop}{Proposition}
\newtheorem{cor}{Corollary}
\newtheorem{lem}{Lemma}
\newtheorem{defi}{Definition}
\newtheorem{rmk}{Remark}
\newtheorem{eg}{Example}
\newcommand\gauss[2]{\begin{bmatrix}#1\\#2\\\end{bmatrix}}

\begin{abstract}
 We prove the monotonic normalized heat diffusion  property on distance-regular graphs with classical parameters of diameter $3$. Regev and Shinkar found a Cayley graph for which this property fails. On the other hand, this property has been proved on abelian Cayley graphs, graphs with $3$ distinct eigenvalues and regular bipartite graphs with $4$ distinct eigenvalues by Price, Nica and Kubo-Namba, respectively. A distance regular graph with classical parameters of diameter $3$ has $4$ distinct eigenvalues and is not necessarily bipartite or vertex transitive. 
\end{abstract}

\section{Introduction}

In this paper, we study the heat kernel $H:[0,\infty)\times V\times V\to \mathbb{R}$ on a simple finite connected graph $G=(V,E)$.
For any two vertices $u,v\in V$ and any $t\geq 0$,  $H_t(u, v)$ can be considered as the probability that a  continuous-time random walker starting at vertex $u$, is at vertex $v$ at time $t$.
For any two distinct vertices $u,v\in V$, we consider the function 
$$t\mapsto r_t(u, v):=\frac{H_t(u, v)}{H_t(u, u)}.$$
Clearly, we have  
\begin{equation}\label{eq:two_ends}
    r_0(u, v)=0, \,\,\text{and}\,\, \lim_{t\to\infty}r_t(u, v)=1.
\end{equation}
It is natural to ask whether the above function is monotonically non-decreasing or not. This question has been considered by Regev and Shinkar \cite{RS16}. We will refer to this property as
the {\it monotonic normalized heat diffusion property}, or the MNHD property for short, following \cite{KN22,NB20}.

Due to \eqref{eq:two_ends}, a necessary condition for the MNHD property is that $r_t(u, v) \le 1$ holds for any $t \ge 0$. However, there do exist graphs for which the function $r_t(u, v)$ exceeds 1 for certain vertices $u,v$ at certain time $t$. This is observed by Regev and Shinkar \cite[Proposition A.1]{RS16}. Regev and Shinkar mentioned that this observation was brought to their attention by Cheeger. They also provided an explicit example which is a $4$-regular graph with $10$ vertices \cite[Figure 2]{RS16}.

A sufficient condition to ensure that $r_t(u, v)\le 1$ is a \emph{heat homogeneity}, that is, the heat kernel has constant diagonal, which means that the value of $H_t(u,u)$ is independent of the choice of $u$ (see \cite[Section 1]{RS16} and \cite[Section 1]{NB20}). In fact, the heat homogeneity holds if and only if the graph is walk-regular \cite[Theorem A.1]{NB20}. Recall that a graph is walk-regular if  for every $\ell\geq 2$, the number of closed walks of length $\ell$
starting from and ending at every vertex is constant. Typical walk-regular graphs includes vertex-transitive graphs, distance-regular graphs, and regular graphs with at most $4$ distinct eigenvalues.

Peres asked in 2013 (see \cite[Section 1]{RS16}) whether the MNHD property holds for all \emph{vertex transitive graphs} or not. Regev and Shinkar \cite{RS16} 
showed that there exists a Cayley graph for which the MNHD property fails. This provides a negative answer to Peres' question. On the other hand, Price \cite{P17} proved that the MNHD property does hold for all finite \emph{abelian} Cayley graphs. 


Nica \cite{NB20} showed that the MNHD property holds for graphs with $3$ distinct Laplacian eigenvalues. We stress that the result of Nica  does not require the graph to be regular.  Kubo and Namba \cite{KN22} confirmed the MNHD property for regular \emph{bipartite} graphs with $4$ distinct eigenvalues. It is natural to ask whether the MNHD property holds for all \emph{regular graphs with at most $4$ distinct eigenvalues}.

\emph{Distance-regular graphs} are introduced by Biggs \cite{Biggs} as a natural extension of distance-transitive graphs. A graph is distance-regular if for any two vertices $u,v$ at distance $h$, the number of vertices that are at distance $i$ to $u$ and at distance $j$ to $v$ only depends on the three numbers $h,i$ and $j$. Every distance-regular grpah has diameter $D$ if and only if it has $D+1$ distinct eigenvalues. 

For a regular graph $G$ with $4$ distinct eigenvalue, 
van Dam and Haemers \cite{DH97} showed that $G$ is distance-regular if and only if the number of vertices at distance $2$ from each given vertex satisfies an expression in terms of eigenvalues of $G$.

We identify a class of distance-regular graphs with $4$ distinct eigenvalues  that are not necessarily \emph{bipartite}, for which the MNHD property holds true.  Our main result states as follows.

\begin{thm}\label{thm1}
    Let  $G=(V,E)$ be a  distance-regular graph with classical parameters  of diameter $3$. Then $G$ has the monotonic normalized heat diffusion (MNHD) property.
\end{thm}

Most distance-regular graphs related to classical groups and groups of Lie type have classical parameters. Lots of interesting graphs constructed from algebraic structures are either distance-regular graphs with  classical  parameters (see, e.g., \cite{DK05} and \cite[Tables 6.1 and 6.2]{BCN89}) or are closely related to distance-regular graphs with  classical  parameters (see, e.g., \cite{BP12} and \cite[Section 3.2]{DKT16}). We refer to Subsection \ref{sec2.2} for the precise definition of a distance-regular graph with classical parameters.
Although there exist various classification results for distance-regular graphs with classical parameters of large diameter (see, e.g. \cite{DKT16}), relatively fewer results are known for the diameter $3$ case. We believe that our investigation into the monotonic normalized heat diffusion (MNHD) property sheds new light on the structure of distance-regular graphs with classical parameters.

In \cite{Biggs82},  Biggs classified the distance-regular graphs of diameter 3 into three categories: antipodal, bipartite and primitive. We show in Section \ref{sec7} that the MNHD property holds for all \emph{antipodal} distance-regular graphs of diameter $3$. Notice that the bipartite case follows from the result of Kubo and Namba \cite{KN22}.

\subsection*{Relevance of the MNHD property in other areas}

It is very natural to study whether the MNHD property holds on a given graph. In general, it is highly concerned to understand how certain properties of solutions to differential/difference equations depend on the structure of the underlying graphs. 
Moreover, the MNHD property is very useful and relevant to other research areas. 

It is also very interesting to study the MNHD property in the setting of Riemannian geometry. Similarly, the MNHD property does not always hold for Riemannian manifolds \cite[Appendix A]{P17}. In the other direction, it was shown that the MNHD property holds on all flat tori \cite{RS16} and some other Riemannian symmetric spaces including the $3$-dimensional hyperbolic space \cite[Section 5 and Appendix B]{P17}.
It is still open whether the MNHD property holds on all Riemannian symmetric spaces of compact type and on all hyperbolic spaces.

In a lattice in $\mathbb{R}^n$, the MNHD property reduces to a monotonicity property of the periodic Gaussian function. An extension of this monotonicity property to the non-spherical Gaussian case \cite[Proposition 4.2]{RSD17} can be used to derive a monotonicity property of the normalized heat diffusion on a torus with respect to different flat Riemannian metrics \cite{P14b, RSD17}. Such an application also leads to a key ingredient in the framework of the so-called numerical cohomology theory \cite{P14a, P17b}.


\section{Preliminaries}\label{sec2}
In this section, we collect basics about graph Laplacian, heat kernel, and distance-regular graphs with classical parameters.

\subsection{Graph Laplacian and heat kernels}\label{sec2.1}

Let $d\ge1$ be an integer and $G=(V,E)$ a simple finite connected graph with vertex set $V$ and edge set $E$. Suppose that $|V|=n\ge 2$. The graph Laplacian $L$ is a symmetric $n\times n$ matrix $L=(L(x,y))_{x,y\in V}$ defined by $L=\mathrm{Deg}-A$, where $\mathrm{Deg}$ is the diagonal vertex degree matrix and $A$ is the adjacency matrix. The constant vector $\mathbf{1}=(1,1,...,1)^{T}$ is an eigenvector of $L$ corresponding to the eigenvalue $0$. Since $G$ is connected, the eigenvalue $0$ is simple by the Perron-Frobenius theorem. Let us denote by $\mathrm{sp}(L)$ the set of all distinct eigenvalues of $L$. The graph Laplacian has the following spectral decomposition
\begin{equation}\label{eq:L_spectral}
    L=\sum_{\lambda \in \mathrm{sp}(L)}\lambda P_\lambda,
\end{equation}
where $P_\lambda$ is
the projection operator onto the eigenspace
associated with $\lambda$. The projection operator $P_0=\langle\cdot,\frac{\mathbf{1}}{\sqrt{n}}\rangle\frac{\mathbf{1}}{\sqrt{n}}$
corresponding to the simple eigenvalue $0$ is an averaging operator, where we use $\langle\cdot, \cdot\rangle$ for the standard Euclidean inner product. In matrix form, we have $P_0=\frac{1}{n}J$, with $J$ being the $n\times n$ all-$1$ matrix.

The heat kernel $H_t$ on the graph $G$ is defined in matrix form as $H_t:=e^{-tL}$. Accordingly, $H_t$ has the following spectral decomposition
\begin{equation}\label{eq:heat_spectral}
    H_t=\sum_{\lambda \in \mathrm{sp}(L)} e^{-t\lambda}P_\lambda\ ,
\end{equation}
 for all $t\ge0$. For any two distinct vertices $u, v \in V$, we define the function 
$h_{u, v} : [0, \infty) \to \mathbb{R}$ as follows, 
\begin{equation}\label{eq:huvt1}
   h_{u, v}(t):=
H_t'(u, v)H_t(u, u) - H_t(u, v)H_t'(u, u). 
\end{equation}

Here $H_t'$ stands for the derivative of $H_t$
with respect to $t$.
The MNHD property is equivalent to  that $h_{u,v}(t)\geq 0$ for any $t\geq 0$. Nica observed that this is always true at $t=0$.
\begin{lem}[{\cite[Lemma 3.2]{NB20}}]\label{lemma:huv0}
 For any two distinct vertices $u,v$, it holds that  $ h_{u, v}(0)\ge 0$.
\end{lem}

Let us list all $\lambda\in \mathrm{sp}(L)$ as follows:
\begin{equation}\label{eq:eig_list}
    0<\lambda_1<\lambda_2<\cdots<\lambda_s.
\end{equation}
Then, we have the following expression of $h_{u,v}(t)$, see, e.g.,  \cite[(3.2)]{KN22}.
\begin{lem}\label{lemma:huvt}
For any two distinct vertices $u,v$ and $t\geq 0$, we have
\begin{equation}\label{eq:huvt}
    h_{u,v}(t)=\sum_{i=1}^s\frac{\lambda_i}{n}e^{-t\lambda_i}\Delta_i(u,v)+\sum_{1\leq i<j\leq s}(\lambda_j-\lambda_i)e^{-t(\lambda_i+\lambda_j)}\Delta_{ij}(u,v),
\end{equation}
where for any $i,j=1,\ldots,s$,
\begin{align*}
\Delta_i(u,v)&:=P_{\lambda_i}(u,u)-P_{\lambda_i}(u,v),\\
\Delta_{ij}(u,v)&:=P_{\lambda_i}(u,v)P_{\lambda_j}(u,u)-P_{\lambda_j}(u,v)P_{\lambda_i}(u,u).
\end{align*}
\end{lem}
\begin{proof}
    This follows from \eqref{eq:heat_spectral} by a straightforward calculation.
\end{proof}

\subsection{Heat kernels on regular graphs with four distinct eigenvalues}\label{sec2.3}
We consider a simple finite connected $d$-regular graph $G=(V,E)$ with $4$ distinct Laplacian eigenvalues, i.e., $s=3$ in \eqref{eq:eig_list}.
In this particular setting, we have the following explicit expressions for the functions $\Delta_i(u,v)$ and $\Delta_{ij}(u,v)$ due to Kubo and Namba \cite{KN22}.

\begin{lem}[{\cite[Lemmas 3 and 4]{KN22}}]\label{lemma:Deltaijs}
    Let $G=(V,E)$ be a simple finite connected $d$ regular graph with $4$ distinct Laplacian eigenvalues $0<\lambda_1<\lambda_2<\lambda_3.$
Denote by \begin{align*}         C_1=\frac{1}{(\lambda_2-\lambda_1)(\lambda_3-\lambda_1)},\,\, C_2 = \frac{1}{(\lambda_3-\lambda_2)(\lambda_1-\lambda_2)},\,\,\text{and}\,\,\,
C_3= \frac{1}{(\lambda_1-\lambda_3)(\lambda_2-\lambda_3)}.
    \end{align*}
    Then, for any two distinct vertices $u, v \in V$, we have
    \begin{align}
    \Delta_1(u, v)
    &=C_1
    \{ d^2+d - L^2(u, v) +(L(u,v)- d)(\lambda_2+\lambda_3)
   +\lambda_2 \lambda_3\},\label{D1}
\\
        \Delta_2(u, v) &=C_2
    \{ d^2+d - L^2(u, v)+(L(u,v) - d)(\lambda_1+\lambda_3)
   +\lambda_1 \lambda_3\}, \label{D2}
\\
    \Delta_3(u, v)&=C_3
    \{ d^2+d - L^2(u, v) +(L(u,v)- d)(\lambda_1+\lambda_2)
    +\lambda_1 \lambda_2\},\label{D3}
\end{align}
and for $\{i,j,k\}=\{1,2,3\}$,
    \begin{align}
        \Delta_{ij}(u, v)
        =C_iC_j(\lambda_j-\lambda_i)
        &\bigg\{\left(d+\frac{1-n}{n}\lambda_k\right)L^2(u,v)   \notag
        \\&-\left(d^2+d+\frac{1-n}{n}\lambda_k ^2\right)L(u,v) 
        -\frac{1}{n}\left((d^2+d)\lambda_k-d\lambda_k ^2\right)\bigg\}.\label{Dij}
    \end{align}
    \end{lem}
\begin{proof}
  The expressions \eqref{D1}-\eqref{D3} are obtained in \cite[Proof of Lemma 3]{KN22} and equation \eqref{Dij} is a rearrangement of  \cite[(3.4)]{KN22}. For readers' convenience, we briefly recall  the key steps of the proof. Solving the equations $I=P_0+P_{\lambda_1}+P_{\lambda_2}+P_{\lambda_3}$ and $L^k=\lambda_1^kP_{\lambda_1}+\lambda_2^kP_{\lambda_2}+\lambda_3^kP_{\lambda_3}$, $k=1,2$ leads to the following expressions for the projection operators:
  \begin{align*}
P_{\lambda_1}
&= C_1
\big\{L^2-(\lambda_2+\lambda_3)L
+\lambda_2\lambda_3(I-P_0)\big\}, \\
P_{\lambda_2}
&= C_2
\big\{ L^2-(\lambda_1+\lambda_3)L
+\lambda_1\lambda_3(I-P_0)\big\}, \\
P_{\lambda_3}
&= C_3
\big\{ L^2-(\lambda_1+\lambda_2)L
+\lambda_1\lambda_2(I-P_0)\big\}.
\end{align*}
Then Lemma \ref{lemma:Deltaijs} follows by a straightforward calculation, noticing that $L(u,u)=d$ and $L^2(u,u)=d^2+d$.
\end{proof}

\subsection{ Distance-regular graph with classical parameters}\label{sec2.2}

A connected graph $G=(V,E)$ of diameter $D$ is
{\em distance-regular} if and only if it is regular of degree $d$, and if there exist constants $c_i, a_i, b_i$, known as  intersection  numbers, such that for all
$i=0,1,\dots, D$, and  for any two vertices  $x,y \in V$ at distance
$i=d(x,y)$, among the neighbors of $y$, there are $c_i$ at
distance $i-1$ from $x$, $a_i$ at distance $i$, and $b_i$ at
distance $i+1$. It follows that $d=b_0$, and $a_i=d-b_i-c_i$ for all $i=0,1,\dots,
D$.  These constants are commonly listed in
the so-called {\em intersection array}
$$\{b_0,b_1,\dots,b_{D-1};c_1,c_2,\dots,c_D\}.$$

Observe that $b_D=0$ and $c_0=0$ are excluded from this array, while $c_1=1$
is included, and all numbers in the intersection array are positive
integers. Furthermore, the number of vertices can be derived from the intersection
array.  Indeed, by counting edges $yz$ with $d(x,y)=i$ and $d(x,z)=i+1$  for each vertex $x$, it has a constant number of vertices $n_i$ at a given
distance $i$. Specifically, it follows from $n_0=1$ and
$n_{i+1}=\frac{b_in_i}{c_{i+1}}$ for all $i=0,1,\dots,D-1$ that the number of vertices $n=n_0+n_1+ \cdots + n_D$.

 The \emph{distance-regular graphs with  classical  parameters} $(D,b,\alpha,\beta)$ are the distance-regular graphs whose intersection numbers can be expressed by the four
parameters: diameter $D$ and numbers $b$, $\alpha$,
$\beta$, in the following way,

\begin{align}
b_{i}&=\left(
\begin{bmatrix}
    D\\
    1\\
\end{bmatrix}
-
\begin{bmatrix}
    i\\
    1\\
\end{bmatrix}
\right)
\left( \beta-\alpha
\begin{bmatrix}
    i\\
    1\\
\end{bmatrix}
\right)\;\;, i=0,1,\dots,D-1,\label{eq:bi}\\
c_{i}&=
\begin{bmatrix}
    i\\
    1\\
\end{bmatrix}
\left(1+\alpha
\begin{bmatrix}
   i-1\\
   1\\
\end{bmatrix}
\right)\;\;, i=1,2,\dots,D,\label{eq:ci}
\end{align}
where $ \begin{bmatrix} j\\1\\\end{bmatrix} = 1+b+b^2 +\cdots +b^{j-1} $ is the  Gaussian binomial coefficient with basis $b$. To make $b_i,c_i$ feasible parameters, there are necessary restrictions on the numbers $b,\alpha, \beta$. For example, 
the number $b$ must be an integer not equal to $0$ or $-1$ \cite[Proposition 6.2.1]{BCN89}. See also Lemma \ref{lem1} below.


\begin{prop}[{\cite[Corollary 8.4.2]{BCN89}}]\label{prop:eig}
Let $G$ be a distance-regular graph with classical parameters $(D,b,\alpha,\beta)$. Then the eigenvalues of its adjacent matrix are
$$\theta_i=\frac{b_i}{b^i}-\gauss{i}{1}, \ \ i=0,1,\cdots,D,$$
where $b_i,i=0,1,\ldots, D$ are given by \eqref{eq:bi}. 
\end{prop}

For more discussion on distance-regular graphs, we refer to \cite{BCN89} and \cite{DKT16}.

\subsection{Distance-regular graphs with classical parameters of diameter 3}\label{sec3}
In this subsection, we collect basics of distance-regular graphs with classical parameters $(D,b,\alpha,\beta)$ of diameter $D=3$. 

First, we derive from \eqref{eq:bi} and \eqref{eq:ci} the intersection numbers as follows.
\begin{equation}
   \label{ia} \begin{array}{ll}
        b_0=d=(1+b+b^2)\beta, &  b_1=(b+b^2)(\beta-\alpha),\\
        b_2=b^2(\beta-\alpha(1+b)),  & b_3=0,\\
        c_0=0, & c_1=1,\\
        c_2=(1+b)(1+\alpha), &  c_3=(1+b+b^2)(1+\alpha(1+b)),\\
        a_0=0, & a_1=-1 + b\alpha + b^2 \alpha + \beta,\\
        a_2=(1 + b) (-1 - \alpha + b^2 \alpha + \beta), & a_3=(1 + b + b^2) (-1-\alpha-b\alpha+\beta).\\
    \end{array}
\end{equation}
Then, we calculate the vertex number $n$ as below,
\begin{align}
 \label{n}  
 n=&1+d+\frac{db_1}{c_2}+\frac{db_1b_2}{c_2c_3}\notag\\
 =&1+d+\frac{db(\beta-\alpha)}{1+\alpha}+\frac{db(\beta-\alpha)b_2}{(1+\alpha)c_3}.
\end{align}
Applying Proposition \ref{prop:eig}, we obtain the eigenvalues of its adjacent matrix as follows,
\begin{equation}\label{eq:adj_eig}
    \theta_0=d, \theta_1=(1+b)(\beta-\alpha)-1, \theta_2=\beta-(1+b)(1+\alpha),\,\text{and}\,\,\theta_3=-(1+b+b^2),
\end{equation}
and, accordingly, the non-zero Laplacian eigenvalues $\gamma_i=d-\theta_i$ listed below,
\begin{align}\label{eq:eigenvalues}
\gamma_1=1+\alpha+b\alpha+b^2\beta,
\gamma_2=(1+b)(1+\alpha+b\beta),\,\text{and}\,\,
\gamma_3=(1+b+b^2)(1+\beta).
\end{align}

In the following, we derive some necessary constraints from \eqref{ia} on the parameters $b,\alpha,\beta$ in case $D=3$.
\begin{lem}\label{lem1}
    Let $G=(V,E)$ be a distance-regular graph with classical parameters $(3,b,\alpha,\beta)$. Then we have the following constraints:
\begin{itemize}
    \item [(i)] If $b\ge1$, then $\alpha\geq 0$.
    \item[(ii)]  If $b\le-2$, then $\alpha<-1$.
    \item [(iii)]  $\beta\ge1+(1+b)\alpha\geq 1$.
\end{itemize}
\end{lem}
\begin{proof}
    Since $c_2=(1+b)(1+\alpha)>0$, we obtain $\alpha<-1$ if $b\leq -2$. Moreover, since $b$ is an integer, we read from $c_2$ that $(1+b)\alpha$ is an integer. Via $c_3=(1+b+b^2)(1+\alpha(1+b))>0$, we derive that $\alpha(1+b)>-1$ if $b\geq 1$.  Therefore, we have $\alpha(1+b)\geq 0$ if $b\geq 1$. This proves $(i)$. The constraint $\beta\geq 1+(1+b)\alpha$ follows from the facts $a_3=(1+b+b^2)(-1-\alpha-b\alpha+\beta)\geq 0$. Applying $(i)$ and $(ii)$, we further deduce $\beta\geq 1+(1+b)\alpha\geq 1$. This completes the proof.
\end{proof}

\begin{rmk} Lemma \ref{lem1} $(i)$ is taken from Lv and Koolen \cite{LK24}. In fact, Lemma \ref{lem1} $(i)$ and $(ii)$ hold for any distance-regular graph with classical parameters of diameter $D\geq 3$. The proof of Lemma \ref{lem1} $(i)$ and $(ii)$ relies only on the expressions of $c_2$ and $c_3$ in \eqref{ia}, which holds for any diameter $D\geq 3$.
\end{rmk}

\section{The MNHD property for regular graphs with four distinct eigenvalues}\label{sec4}

In this section, we derive a sufficient condition for the MNHD property to hold on regular graphs with four distinct eigenvalues. 

Our main Theorem \ref{condition} in this section relies on the following key observation. 
\begin{lem}\label{identity}
For any three distinct numbers $\lambda_i, \lambda_j, \lambda_k \in \mathbb{R}$ with $\{i,j,k\}=\{1,2,3\}$ and any numbers $L(u,v), L^2(u,v)\in \mathbb{R}$, let $\Delta_i(u,v), i=1,2,3$ and $\Delta_{ij}(u,v)$ be defined in terms of $\lambda_i, \lambda_j,\lambda_k, L(u,v), L^2(u,v)$ as in \eqref{D1}-\eqref{Dij}. Then it holds that
\begin{align}
&\frac{\lambda_i\lambda_j}{n}\Delta_i(u,v) + \frac{\lambda_j\lambda_i}{n}\Delta_j(u,v) + \frac{\lambda_k(\lambda_i + \lambda_j - \lambda_k)}{n}\Delta_k(u,v)\notag\\
&\hspace{6cm}- (\lambda_k - \lambda_i)(\lambda_k - \lambda_j)(\Delta_{ik}(u,v)+\Delta_{jk}(u,v))\notag \\
=& L^2(u,v) - L(u,v)(\lambda_i + \lambda_j). \label{eq:identity}
\end{align}
\end{lem}
\begin{rmk}
    Notice that in identity \eqref{eq:identity}, the three distinct numbers $\lambda_i, \lambda_j, \lambda_k \in \mathbb{R}$ are not required to be eigenvalues of the graph Laplacian $L$; the two numbers $L(u,v)$ and $L^2(u,v)$ are not required to be obtained from the graph Laplacian $L$ and its square $L^2$. We only use the formal definitions of $\Delta_i$ and $\Delta_{ij}$ in \eqref{D1}-\eqref{Dij}. 
\end{rmk}
\begin{proof}
We break down the calculations into two parts.
First, we consider the terms involving \(\Delta_i(u,v)\), \(\Delta_j(u,v)\), and \(\Delta_k(u,v)\). We calculate by inserting \eqref{D1}-\eqref{D3} that 
\begin{align}
&\frac{\lambda_i \lambda_j}{n} \Delta_i(u,v) + \frac{\lambda_j \lambda_i}{n} \Delta_j(u,v) + \frac{\lambda_k(\lambda_i + \lambda_j - \lambda_k) }{n} \Delta_k(u,v) \notag \\
=& \frac{\lambda_i \lambda_j \left(d + d^{2} - L^2(u,v) +(L(u,v) - d)(\lambda_j + \lambda_k) + \lambda_j \lambda_k\right)}{n(\lambda_i - \lambda_j)(\lambda_i - \lambda_k)} \notag \\
&+ \frac{\lambda_i \lambda_j \left(d + d^{2} - L^2(u,v)  +(L(u,v)- d)(\lambda_i + \lambda_k) + \lambda_i \lambda_k\right)}{n(\lambda_j - \lambda_i)(\lambda_j - \lambda_k)} \notag \\
&+ \frac{\lambda_k (\lambda_i + \lambda_j - \lambda_k) \left(d + d^2 - L^2(u,v)+(L(u,v)- d)(\lambda_i + \lambda_j)  + \lambda_i \lambda_j \right)}{n(\lambda_k - \lambda_i)(\lambda_k - \lambda_j)} \notag \\
=&\frac{(d+d^2-L^2(u,v))\mathrm{I}+(L(u,v)-d)\mathrm{II}+\mathrm{III}}{n(\lambda_i-\lambda_j)(\lambda_i-\lambda_k)(\lambda_j-\lambda_k)},\notag
\end{align}
where we have
\begin{align*}
\mathrm{I}:=&\lambda_i\lambda_j(\lambda_j-\lambda_k)-\lambda_i\lambda_j(\lambda_i-\lambda_k)+\lambda_k(\lambda_i+\lambda_j-\lambda_k)(\lambda_i-\lambda_j)\\
=&
-(\lambda_i-\lambda_j)(\lambda_i-\lambda_k)(\lambda_j-\lambda_k),\\
    \mathrm{II}:=&\lambda_i\lambda_j(\lambda_j^2-\lambda_k^2)-\lambda_i\lambda_j(\lambda_i^2-\lambda_k^2)+\lambda_k(\lambda_i+\lambda_j-\lambda_k)(\lambda_i^2-\lambda_j^2)\\
    =&
    -(\lambda_i-\lambda_j)(\lambda_i-\lambda_k)(\lambda_j-\lambda_k)(\lambda_i+\lambda_j),\\
\mathrm{III}:=&\lambda_i\lambda_j^2\lambda_k(\lambda_j-\lambda_k)-\lambda_i^2\lambda_j\lambda_k(\lambda_i-\lambda_k)+\lambda_i\lambda_j\lambda_k(\lambda_i+\lambda_j-\lambda_k)(\lambda_i-\lambda_j)\\=&0.
\end{align*}
Putting together yields that
\begin{align}
   &\frac{\lambda_i \lambda_j}{n} \Delta_i(u,v) + \frac{\lambda_j \lambda_i}{n} \Delta_j(u,v) + \frac{\lambda_k(\lambda_i + \lambda_j - \lambda_k) }{n} \Delta_k(u,v) \notag \\
   =&-\frac{1}{n}\left\{d+d^2-L^2(u,v)+(L(u,v)-d)(\lambda_i+\lambda_j)\right\}.\label{eq:Delta_part}
\end{align}

Next, we consider the terms involving \(\Delta_{ik}(u,v)\) and \(\Delta_{jk}(u,v)\). By inserting \eqref{Dij}, we calculate that
\begin{align}
&- (\lambda_k - \lambda_i)(\lambda_k - \lambda_j) (\Delta_{ik}(u,v)+\Delta_{jk}(u,v)) \notag \\
=&\frac{(nd+(1-n)\lambda_j)L^2(u,v)-(n(d^2+d)+(1-n)\lambda_j^2)L(u,v)-(d^2+d)\lambda_j+d\lambda_j^2}{n(\lambda_i-\lambda_j)}\notag\\
&+\frac{(nd+(1-n)\lambda_i)L^2(u,v)-(n(d^2+d)+(1-n)\lambda_i^2)L(u,v)-(d^2+d)\lambda_i+d\lambda_i^2}{n(\lambda_j-\lambda_i)}\notag\\
=&\frac{1}{n}\left\{-(1-n)\left(L^2(u,v)-L(u,v)(\lambda_i+\lambda_j)\right)+d^2+d-d(\lambda_i+\lambda_j)\right\}.\label{eq:Deltaij_part}
\end{align}

Combining \eqref{eq:Delta_part} and \eqref{eq:Deltaij_part} leads to the identity \eqref{identity}.
\end{proof}

Next, we establish a sufficient condition for a regular graph with four distinct eigenvalues to have MNHD property.
\begin{thm}\label{condition}
Let $G$ be a regular graph with four distinct Laplacian eigenvalues \[0 < \lambda_1 < \lambda_2 < \lambda_3.\] Let $u,v$ be two distinct vertices. Suppose that $\Delta_i(u, v)\geq 0, i=1,2,3$, and one of the following conditions holds,
\begin{itemize}
    \item [(i)] $L^2(u,v)-L(u,v)(\lambda_2+\lambda_3)\geq 0$, and $\Delta_{12}(u, v) \geq 0$,
    \item [(ii)] $L^2(u,v)-L(u,v)(\lambda_1+\lambda_2)\geq 0$, $\Delta_{13}(u, v) \leq 0$, and $\lambda_1+\lambda_2-\lambda_3 \geq 0$,
\item [(iii)] $L^2(u,v)-L(u,v)(\lambda_1+\lambda_3)\geq 0$, $\Delta_{23}(u, v) \geq 0$, and $\lambda_1+\lambda_2-\lambda_3 \geq 0$.
\end{itemize}
Then the function $r_t(u,v):=\frac{H_t(u,v)}{H_t(u,u)}$ is monotonically non-decreasing in $t$.
\end{thm}

\begin{proof}
Recall that the monotonic property of $r_{t}(u,v)$ in $t$ is equivalent to the non-negativity of the function $h_{u,v}(t)$ defined in \eqref{eq:huvt1} and \eqref{eq:huvt}. 

\textbf{Case  $(i)$}:
Consider the function $F(t):=e^{t(\lambda_2+\lambda_3)}h_{u,v}(t)$. Employing Lemma \ref{lemma:huvt}, we derive
    \begin{align*}
       F(t)=&\frac{\lambda_1}{n}\Delta_1(u,v)e^{t(\lambda_2+\lambda_3-\lambda_1)}+\frac{\lambda_2}{n}\Delta_2(u,v)e^{t\lambda_3}+\frac{\lambda_3}{n}\Delta_3(u,v)e^{t\lambda_2}\\&+(\lambda_2-\lambda_1)\Delta_{12}(u,v)e^{t(\lambda_3-\lambda_1)}+(\lambda_3-\lambda_1)\Delta_{13}(u,v)e^{t(\lambda_2-\lambda_1)}\\&+(\lambda_3-\lambda_2)\Delta_{23}(u,v).
    \end{align*}
Then, we calculate 
\begin{align*}
       F'(t)=&\frac{\lambda_1}{n}\Delta_1(u,v)(\lambda_2+\lambda_3-\lambda_1)e^{t(\lambda_2+\lambda_3-\lambda_1)}+\frac{\lambda_2\lambda_3}{n}\Delta_2(u,v)e^{t\lambda_3}\\&+\frac{\lambda_3\lambda_2}{n}\Delta_3(u,v)e^{t\lambda_2}+(\lambda_3-\lambda_1)(\lambda_2-\lambda_1)\Delta_{12}(u,v)e^{t(\lambda_3-\lambda_1)}\\&+(\lambda_2-\lambda_1)(\lambda_3-\lambda_1)\Delta_{13}(u,v)e^{t(\lambda_2-\lambda_1)}.
    \end{align*}
Extracting the factor $e^{t(\lambda_2-\lambda_1)}$, we derive
    \begin{align*}
       F'(t)=&e^{t(\lambda_2-\lambda_1)}\bigg(\frac{\lambda_1}{n}\Delta_1(u,v)(\lambda_2+\lambda_3-\lambda_1)e^{t\lambda_3}+\frac{\lambda_3\lambda_2}{n}\Delta_2(u,v)e^{t(\lambda_3-\lambda_2+\lambda_1)}\\&+\frac{\lambda_3\lambda_2}{n}\Delta_3(u,v)e^{t\lambda_1}+(\lambda_3-\lambda_1)(\lambda_2-\lambda_1)\Delta_{12}(u,v)e^{t(\lambda_3-\lambda_2)}\\&+(\lambda_3-\lambda_1)(\lambda_2-\lambda_1)\Delta_{13}(u,v)\bigg).
    \end{align*}
   Since $\Delta_i(u,v)\geq 0, i=1,2,3$ and $\Delta_{12}(u,v)\geq 0$, we estimate 
\begin{align*}
        e^{-t(\lambda_2-\lambda_1)}F'(t)\ge&\frac{\lambda_1(\lambda_2+\lambda_3-\lambda_1)}{n}\Delta_1(u,v)+\frac{\lambda_3\lambda_2}{n}\Delta_2(u,v)+\frac{\lambda_3\lambda_2}{n}\Delta_3(u,v)\\&+(\lambda_3-\lambda_1)(\lambda_2-\lambda_1)(\Delta_{12}(u,v)+\Delta_{13}(u,v))\\=&L^2(u,v) - L(u,v)(\lambda_2 + \lambda_3).
\end{align*}
We have used the fact that $\Delta_{ij}(u,v)=-\Delta_{ji}(u,v)$ and Lemma \ref{identity} in the above estimate. By assumption that $L^2(u,v) - L(u,v)(\lambda_2 + \lambda_3)\geq 0$, we derive
 $$ F(t)\ge F(0)=h_{u,v}(0)\ge0,\  \text{for any}\  t\ge0,$$
 where we used Lemma \ref{lemma:huv0}. 
Therefore, we obtain $h_{u,v}(t)\ge0, \text{for all}\  t\ge0.$

\textbf{Case $(ii)$:} We consider the function $ G(t)=e^{t(\lambda_1+\lambda_2)}h_{u,v}(t)$. Employing Lemma \ref{lemma:huvt}, we derive
    \begin{align*}
        G(t)=&\frac{\lambda_1}{n}\Delta_1(u,v)e^{t\lambda_2}+\frac{\lambda_2}{n}\Delta_2(u,v)e^{t\lambda_1}+\frac{\lambda_3}{n}\Delta_3(u,v)e^{t(\lambda_1+\lambda_2-\lambda_3)}\\&+(\lambda_2-\lambda_1)\Delta_{12}(u,v)+(\lambda_3-\lambda_1)\Delta_{13}(u,v)e^{t(\lambda_2-\lambda_3)}\\&+(\lambda_3-\lambda_2)\Delta_{23}(u,v)e^{t(\lambda_1-\lambda_3)}.
    \end{align*}
Similarly as in Case $(i)$, we estimate
\begin{align*}
    e^{-t(\lambda_1-\lambda_3)}G'(t)
    \ge&\frac{\lambda_1\lambda_2}{n}\Delta_1(u,v)+\frac{\lambda_2\lambda_1}{n}\Delta_2(u,v)+\frac{\lambda_3(\lambda_1+\lambda_2-\lambda_3)}{n}\Delta_3(u,v)\\&+(\lambda_3-\lambda_1)(\lambda_2-\lambda_3)(\Delta_{13}(u,v)+\Delta_{23}(u,v))\\=&L^2(u,v) - L(u,v)(\lambda_1 + \lambda_2).
\end{align*}
The last identity is derived from Lemma \ref{identity}. In particular, we mention that our assumption $\lambda_1+\lambda_2-\lambda_3\geq 0$ is used here.

Therefore, we deduce that
 $$ G(t)\ge G(0)=   h_{u,v}(0)\ge0,\  \text{for any}\  t\ge0.$$
 Hence, we have $h_{u,v}(t)\ge0$, for any $t\ge0$.

\textbf{Case $(iii)$:} We consider the function $ H(t)=e^{t(\lambda_1+\lambda_3)}h_{u,v}(t)$. Employing Lemma \ref{lemma:huvt}, we derive
    \begin{align*}
        H(t)=&\frac{\lambda_1}{n}\Delta_1(u,v)e^{t\lambda_3}+\frac{\lambda_2}{n}\Delta_2(u,v)e^{t(\lambda_1+\lambda_3-\lambda_2)}+\frac{\lambda_3}{n}\Delta_3(u,v)e^{t\lambda_1}\\&+(\lambda_2-\lambda_1)\Delta_{12}(u,v)e^{t(\lambda_3-\lambda_2)}+(\lambda_3-\lambda_1)\Delta_{13}(u,v)\\&+(\lambda_3-\lambda_2)\Delta_{23}(u,v)e^{t(\lambda_1-\lambda_2)}.
    \end{align*}
Similarly as in Case $(ii)$, we estimate
\begin{align*}
   &e^{-t(\lambda_3-\lambda_2)} H'(t)\\=& \frac{\lambda_1 \lambda_3}{n} \Delta_1(u,v) e^{t \lambda_2} + \frac{\lambda_2 (\lambda_1 + \lambda_3 - \lambda_2)}{n} \Delta_2(u,v) e^{t \lambda_1}
+ \frac{\lambda_3 \lambda_1}{n} \Delta_3(u,v) e^{t(\lambda_1 + \lambda_2 - \lambda_3)} \\ &+ (\lambda_3 - \lambda_2)(\lambda_1 - \lambda_2) \Delta_{23}(u,v) e^{t(\lambda_1 - \lambda_3)}
+ \left. (\lambda_2 - \lambda_1)(\lambda_3 - \lambda_2) \Delta_{12}(u,v)\right)\\ \ge&\frac{\lambda_1\lambda_3}{n}\Delta_1(u,v)+\frac{\lambda_2(\lambda_1+\lambda_3-\lambda_2)}{n}\Delta_2(u,v)+\frac{\lambda_3\lambda_1}{n}\Delta_3(u,v)\\&+(\lambda_3-\lambda_2)(\lambda_1-\lambda_2)\Delta_{23}(u,v)+(\lambda_2-\lambda_1)(\lambda_3-\lambda_2)\Delta_{12}(u,v)\\=&L^2(u,v) - L(u,v)(\lambda_1 + \lambda_3).
\end{align*}
Notice that we have used $\lambda_1+\lambda_2-\lambda_3\geq 0$ in the inequality above and  Lemma \ref{identity} in the equality. It follows that
 $$ H(t)\ge H(0)=   h_{u,v}(0)\ge0,\  \text{for any}\  t\ge0.$$
 Hence, we have $h_{u,v}(t)\ge0$, for any $t\ge0$.
 This completes the proof.
\end{proof}

\begin{rmk}
Observe from $L=dI-A$ that
\begin{align}\label{1}
    L^2(x, y)
   \notag &=d^2I(x,y)-2dA(x,y)+A^2(x,y)\\ 
    &=\begin{cases}
        d^2+d, & \text{if $x=y$,} \\
        -2dA(x, y)
        +\sum_{w \neq x, y}A(x, w)A(w, y), &
        \text{if $x \neq y$.}
        \end{cases}
\end{align}
Therefore, for any non-adjacent vertices $u,v$, the condition $L^2(u,v)-L(u,v)(\lambda_i+\lambda_j)\geq 0$ in Theorem \ref{condition} always holds. 
\end{rmk}
\begin{rmk} Let $G$ be a regular bipartite graph with four distinct eigenvalues $0<\lambda_1<\lambda_2<\lambda_3$. Then the conditions $L^2(u,v)-L(u,v)(\lambda_i+\lambda_j)\geq 0$ and $\lambda_1+\lambda_2-\lambda_3\geq 0$ in Theorem \ref{condition} all holds true. Indeed, we have $L^2(u,v)=-2d$ for adjacent vertices $u$ and $v$ from \eqref{1}, and 
\[\lambda_1=d-\sqrt{d-\lambda}, \lambda_2=d+\sqrt{d-\lambda}, \lambda_3=2d,\]
where $\lambda:=2d(d-1)/(n-2)$ with $n$ being the number of vertices, see \cite[Propostions 3 and 4]{KN22}. Consequently, we obtain $\lambda_i+\lambda_j\geq 2d$ for any $\{i,j\}\in \{1,2,3\}$ and $\lambda_1+\lambda_2=\lambda_3$.
\end{rmk}

\section {Proof of Theorem \ref{thm1}}\label{sec:proof}
In this section, we give the proof of our main Theorem \ref{thm1}.
We first figure out the entries of matrix $L^2$ for a distance-regular graph. 

\begin{prop}\label{prop:L^2}
    Let $G=(V,E)$ be a distance-regular graph. Then we have for any $x,y\in V$ that
\begin{equation}\label{L^2}
L^{2}\left(x,y\right)=
\begin{cases}
    d^{2}+d, &\text{if}\;d\left(x,y\right)=0;\\
    -2d+a_{1}, &\text{if}\;d\left(x,y\right)=1;\\
    c_{2}, &\text{if}\;d\left(x,y\right)=2;\\
    0, &\text{if}\;d\left(x,y\right)\geq 3.\\
\end{cases}
\end{equation}
\end{prop}

\begin{proof}
This follows directly from \eqref{1}. We only need to observe that 
\begin{equation*}
    \sum_{w\neq x,y}A(x,w)A(w,y)=\begin{cases}
    a_1, &\text{if $d(x,y)=1$};\\
    c_2, &\text{if $d\left(x,y\right)=2$};\\
    0, &\text{if $d\left(x,y\right)\geq 3$},
\end{cases}
\end{equation*}
from the distance-regularity. 
\end{proof}
Next, we show two interesting properties of eigenvalues.

\begin{lem}\label{lemma:1+2-3}
    Let $G$ be a distance-regular graph with classical parameters $(3,b,\alpha,\beta)$ and Laplacian eigenvalues 
    $0<\lambda_1<\lambda_2<\lambda_3$.
    Then we have
    \[\lambda_1+\lambda_2-\lambda_3\geq 0.\]
\end{lem}
\begin{proof}
    Recall the eigenvalues $\gamma_1,\gamma_2,\gamma_3$ from \eqref{eq:eigenvalues}. We check directly that 
    \begin{align}
    \gamma_3-\gamma_1&=(\beta+b-\alpha)(b+1),\notag\\
    \gamma_3-\gamma_2&=\beta+b^2-\alpha(b+1),\label{eq:gammas}\\
    \gamma_2-\gamma_1&=b(\beta+1).\notag
    \end{align}
Recall that $b$ is an integer not equal to $0$ or $-1$.
Applying $\beta\geq 1+(1+b)\alpha\geq 1$ from Lemma \ref{lem1}, we deduce that 
\begin{equation}\label{eq:eigen_order}
\begin{cases}
      0<\gamma_1<\gamma_2<\gamma_3, &\text{if $b\geq 1$;}\\
      0<\gamma_2<\gamma_3<\gamma_1, &\text{if $b\leq -2$.}
\end{cases}
\end{equation}
If $b\geq 1$, we know $\beta\geq 1$ and $\alpha\geq 0$ from Lemma \ref{lem1}. Hence, we derive that
\begin{align*}
   \lambda_1+\lambda_2-\lambda_3=\gamma_1+\gamma_2-\gamma_3=2\alpha(1+b)+(b^2-1)(\beta-1)\geq 0.
\end{align*}
If $b\leq -2$, we again have $\beta\geq 1$ and hence
\begin{align*}
    \lambda_1+\lambda_2-\lambda_3=\gamma_3+\gamma_2-\gamma_1=(b+1)^2(\beta+1)\geq 0.
\end{align*}
This completes the proof. 
\end{proof}

\begin{lem}\label{lemma:L2-L(2+3)}
     Let $G$ be a distance-regular graph with classical parameters $(3,b,\alpha,\beta)$ and Laplacian eigenvalues $0<\lambda_1<\lambda_2<\lambda_3$. Then for any two adjacent vertices $u$ and $v$ we have 
    $$L^2(u,v)-L(u,v)(\lambda_2+\lambda_3)\geq0.$$
\end{lem}

\begin{proof}
     If $b\geq 1$, then we derive by \eqref{ia}, \eqref{eq:eigenvalues}, \eqref{eq:eigen_order} and  Lemma \ref{lem1} that
    \begin{align*}
        &L^2(u,v)-L(u,v)(\lambda_2+\lambda_3)
        =-2d+a_1+(\gamma_2+\gamma_3)\\
        =&-2(1+b+b^2)\beta+(-1+b\alpha+b^2\alpha+\beta)+(1+b)(1+\alpha+b\beta)+(1+b+b^2)(1+\beta)\\
        =&(1+b)^2(1+\alpha)>0.
    \end{align*}
 If $b\leq-2$, we deduce similarly that
     \begin{align*}
     &L^2(u,v)-L(u,v)(\lambda_2+\lambda_3)=-2d+a_1+(\gamma_3+\gamma_1)\\
        =&-2(1+b+b^2)\beta+(-1+b\alpha+b^2\alpha+\beta)+(1+b+b^2)(1+\beta)+(1+\alpha+b\alpha+b^2\beta)\\
        =&-b\beta+(1+b)^2\alpha+1+b+b^2\\
        > &-b(1+(1+b)\alpha)+(1+b)^2\alpha+1+b+b^2\\
        =&(1+b)\alpha+1+b^2\\
        >&-(1+b)+1+b^2\\
        =&b^2-b>0.
    \end{align*}
    This completes the proof.
\end{proof}
In the following two subsections, we present the proof of Theorem \ref{thm1}, in case of $b\ge1$ and $b\le-2$, respectively. 

\subsection{Proof of Theorem \ref{thm1} in case \texorpdfstring{$b\geq 1$}{b>=1}}\label{sec5}
 Let $G$ be a distance-regular graph with classical parameters $(3,b,\alpha,\beta)$ and Laplacian eigenvalues 
    $0<\lambda_1<\lambda_2<\lambda_3$. We assume $b\geq 1$ in this subsection. 
By \eqref{eq:eigen_order}, we have 
  \begin{align}\label{lam}
  \lambda_1=1+\alpha+b\alpha+b^2\beta, \lambda_2=(1+b)(1+\alpha+b\beta),\,\text{and}\,\,\lambda_3=(1+b+b^2)(1+\beta).
  \end{align}

\begin{lem} \label{+di}
Let $G$ be a distance-regular graph with classical parameters $(3,b,\alpha,\beta)$, where $b\ge1$. Then we have for any two distinct vertices $u$ and $v$ that
$$\Delta_1(u,v)\ge 0,\Delta_2(u,v)\ge 0, and\  \Delta_3(u,v)\ge 0.$$
\end{lem}
\begin{proof}
    Inserting the expressions of $L^2(u,v)$ in Proposition \ref{prop:L^2} into \eqref{D1}-\eqref{D3} yields for any $i\in \{1,2,3\}$ that
    \begin{equation}\label{eq:drgdij}
        \Delta_i(u,v)=\begin{cases}
      C_i\left\{(\lambda_j-d-1)(\lambda_k-d-1)+d-a_1-1\right\}, &\text{if $d(u,v)=1$;}\\
        C_i\left\{(\lambda_j-d)(\lambda_k-d)+d-c_2\right\}, &\text{if $d(u,v)=2$;}\\
        C_i\left\{(\lambda_j-d)(\lambda_k-d)+d\right\}, &\text{if $d(u,v)=3$,}
\end{cases}
    \end{equation}
    where $\{j,k\}=\{1,2,3\}\setminus\{i\}$. 
    
    Recalling that $C_1>0$, $C_2<0$, $C_3>0$, and $c_2> 0$, it remains to show
    \begin{itemize}
   \item [(i)] $(\lambda_2-d-1)(\lambda_3-d-1)+d-a_1-1\geq 0$, and $(\lambda_2-d)(\lambda_3-d)+d-c_2\geq 0$.
    \item [(ii)] $(\lambda_1-d-1)(\lambda_3-d-1)+d-a_1-1\leq 0$, and $(\lambda_1-d)(\lambda_3-d)+d\leq 0$.
    \item [(iii)] $(\lambda_1-d-1)(\lambda_2-d-1)+d-a_1-1\geq 0$, and $(\lambda_1-d)(\lambda_2-d)+d-c_2\geq 0$.
    \end{itemize}
Notice that $d-\lambda_i=d-\gamma_i=\theta_i$ whose expressions are provided in \eqref{eq:adj_eig}. Next, we check the above inequalities case by case. 

\textbf{Case $(i)$}: Inserting the expression $d-a_1-1=b_1=(b+b^2)(\beta-\alpha)$ from \eqref{ia}, we deduce
\begin{align}
    &(\lambda_2-d-1)(\lambda_3-d-1)+d-a_1-1=(1+\theta_2)(1+\theta_3)+b_1\notag\\
    =&\left\{1+\beta-(1+b)(1+\alpha)\right\}(-b-b^2)+(b+b^2)(\beta-\alpha)\notag\\
    =&b^2(1+b)(1+\alpha)> 0,\label{eq:(1+theta2)(1+theta3)}
\end{align}
where the last inequality follows from the fact that $\alpha\geq 0$ due to Lemma \ref{lem1}. Furthermore, we derive by noticing $d=(1+b+b^2)\beta$ and $c_2=(1+b)(1+\alpha)$ that
\begin{align}
    &(\lambda_2-d)(\lambda_3-d)+d-c_2=\theta_2\theta_3+d-c_2\notag\\
    =&-\left\{\beta-(1+b)(1+\alpha)\right\}(1+b+b^2)+(1+b+b^2)\beta-c_2\notag\\
    =&c_2(b+b^2)> 0.\label{eq:theta2theta3}
\end{align}

\textbf{Case $(ii)$}: Similarly as in Case $(i)$, we derive
\begin{align}
    &(\lambda_1-d-1)(\lambda_3-d-1)+d-a_1-1=(1+\theta_1)(1+\theta_3)+b_1\notag\\
    =&(1+b)(\beta-\alpha)(-b-b^2)+(b+b^2)(\beta-\alpha)\notag\\
    =&-(b+b^2)b(\beta-\alpha)<0,\label{eq:(1+theta1)(1+theta3)}
\end{align}
and
\begin{align}
    &(\lambda_1-d)(\lambda_3-d)+d=\theta_1\theta_3+d\notag\\
    =&-\left\{(1+b)(\beta-\alpha)-1\right\}(1+b+b^2)+(1+b+b^2)\beta\notag\\
    =&(1+b+b^2)(1+(1+b)\alpha-b\beta)\leq 0,\label{eq:theta1theta3}
\end{align}
where we have used $\beta\geq 1+(1+b)\alpha$.

\textbf{Case $(iii)$}: Similarly as in Case $(i)$, we deduce
\begin{align}
    &(\lambda_1-d-1)(\lambda_2-d-1)+d-a_1-1=(1+\theta_1)(1+\theta_2)+b_1\notag\\
    =&(1+b)(\beta-\alpha)\left\{1+\beta-(1+b)(1+\alpha)\right\}+(b+b^2)(\beta-\alpha)\notag\\
    =&(1+b)(\beta-\alpha)(\beta-b\alpha-\alpha)\geq0,\label{eq:(1+theta1)(1+theta2)}
\end{align}
and
\begin{align}
    &(\lambda_1-d)(\lambda_2-d)+d-c_2=\theta_1\theta_2+d-c_2\notag\\
    =&\left\{(1+b)(\beta-\alpha)-1\right\}\left\{\beta-(1+b)(1+\alpha)\right\}+(1+b+b^2)\beta-c_2\notag\\
   =&(1+b)(\beta-\alpha-1)(\beta-b\alpha-\alpha)\geq0,\label{eq:theta1theta2}
\end{align}
where we have used $\beta\geq 1+(1+b)\alpha$ and $\alpha\geq 0$. This completes the proof.
\end{proof}

\begin{lem}\label{+dij}
Let $G$ be a distance-regular graph with classical parameters  $(3,b,\alpha,\beta)$, where $b\ge1$. For any two distinct vertices $u$ and $v$, we have the following estimates.
\begin{itemize}
    \item [(i)] If $d(u,v)=1$, then $\Delta_{12}(u,v)\ge0$.
\item [(ii)]  If $d(u,v)=2$ and $\beta\ge1+(2 + b) \alpha$, then $\Delta_{12}(u,v)\ge0$.
\item[(iii)] If $d(u,v)=2$ and $\beta<1+(2 + b) \alpha$, then $\Delta_{13}(u,v)\le0$.
    \item [(iv)] If $d(u,v)=3$, then $\Delta_{23}(u,v)\ge0$.
\end{itemize}
\end{lem}

\begin{proof}
For $\{i,j,k\}=\{1,2,3\}$, we obtain by rearranging \eqref{Dij} that
\begin{align*}
\frac{n\Delta_{ij}(u,v)}{C_iC_j(\lambda_j-\lambda_i)}=(d+(n-1)L(u,v))\lambda_k^2-&\left((n-1)L^2(u,v)+d^2+d\right)\lambda_k\\
+&nd\left(L^2(u,v)-(d+1)L(u,v)\right).
\end{align*}
Applying Proposition \ref{prop:L^2}, we further deduce that
\begin{align*}
\frac{n\Delta_{ij}(u,v)}{C_iC_j(\lambda_j-\lambda_i)}=
\begin{cases}
d\lambda_k(\lambda_k-d-1), &\text{if $d(u,v)=3$;}\\
d\lambda_k(\lambda_k-d-1)-\left((n-1)\lambda_k-nd\right)c_2, &\text{if $d(u,v)=2$;}\\
-(n-1-d)\lambda_k^2+\left((n-1)(2d-a_1)-d(d+1)\right)\lambda_k &\\
\hspace{5.5cm}-nd(d-a_1-1), &\text{if $d(u,v)=1$.}
\end{cases}
\end{align*}
Inserting $d-a_1-1=b_1$, we have in the case $d(u,v)=1$ that
\begin{align}
    \frac{n\Delta_{ij}(u,v)}{C_iC_j(\lambda_j-\lambda_i)}=&-(n-1-d)\lambda_k^2+\left((n-1-d)(d+1)+(n-1)b_1\right)\lambda_k -ndb_1\notag\\
       =&-(n-1-d)\lambda_k(\lambda_k-d-1)+\left((n-1)\lambda_k-nd\right)b_1.\label{eq:Deltaij_simple}
\end{align}

\textbf{Case $(i)$: $d(u,v)=1$.} 
Recall that $\lambda_3=\gamma_3=(1+b+b^2)(1+\beta)=d+1+b+b^2$. We deduce
\begin{align*}
       \frac{n\Delta_{12}(u,v)}{C_1C_2(\lambda_2-\lambda_1)}
       =&-(n-1-d)(d+1+b+b^2)(b+b^2)+\left((n-1)(1+b+b^2)-d\right)b_1\\
       =&(1+b+b^2)\left\{-(n-1-d)(1+\beta)b(1+b)+(n-1-\beta)b_1\right\}
\end{align*}
Using $b_1=(b+b^2)(\beta-\alpha)$, we arrive at
\begin{align}
       \frac{n\Delta_{12}(u,v)}{C_1C_2(\lambda_2-\lambda_1)}
       =&(1+b+b^2)b(1+b)\left\{-(n-1-d)(1+\beta)+(n-1-\beta)(\beta-\alpha)\right\}\notag\\
       =&(1+b+b^2)b(1+b)\left\{-(n-1)(1+\alpha)+d(1+\beta)-\beta(\beta-\alpha)\right\}.\label{eq:deltaCase1}
\end{align}
By \eqref{n}, we have 
\begin{equation}\label{eq:nalpha}
    (n-1)(1+\alpha)=d(1+\alpha)+db(\beta-\alpha)+db(\beta-\alpha)\frac{b_2}{c_3}.
\end{equation}
Therefore, we deduce
\begin{align*}
    &-(n-1)(1+\alpha)+d(1+\beta)-\beta(\beta-\alpha)\\
    < &-d(1+\alpha)-db(\beta-\alpha)+d(1+\beta)-\beta(\beta-\alpha)\\
    =&(\beta-\alpha)(-db+b-\beta)< 0.
\end{align*}
Inserting into \eqref{eq:deltaCase1}, we obtain 
\[\frac{n\Delta_{12}(u,v)}{C_1C_2(\lambda_2-\lambda_1)}< 0,\]
which implies that $\Delta_{12}(u,v)> 0$.

\textbf{Case $(ii)$: $d(u,v)=2$ and $\beta\geq 1+(2+b)\alpha$.}
We have in this case that 
\begin{align}
    \frac{n\Delta_{12}(u,v)}{C_1C_2(\lambda_2-\lambda_1)}
    =&d\lambda_3(\lambda_3-d-1)-\left((n-1)\lambda_3-nd\right)c_2=:B_1+B_2.\label{eq:deltaCase2}
\end{align}
Recall that $\lambda_3=(1+b+b^2)(1+\beta)$, $d=(1+b+b^2)\beta$ and $c_2=(1+b)(1+\alpha)$. We deduce that
\begin{align*}
    B_1:=d\lambda_3(\lambda_3-d-1)=(1+b+b^2)db(1+b)(1+\beta),
\end{align*}
and
\begin{align*}
    B_2:=-\left((n-1)\lambda_3-nd\right)c_2=-(1+b+b^2)(1+b)(1+\alpha)(n-1-\beta).
\end{align*}
Applying \eqref{n}, we further calculate that
\begin{align*}
    &\frac{B_1+B_2}{(1+b)(1+b+b^2)}=-(n-1-\beta)(1+\alpha)+db(1+\beta)\\
    =&-(d-\beta)(1+\alpha)-db(\beta-\alpha)-db(\beta-\alpha)\frac{b_2}{c_3}+db(1+\beta)\\
    =&(db-d+\beta)(1+\alpha)-db(\beta-\alpha)\frac{b_2}{c_3}\\
    =&b\beta\left\{b^2-(1+b+b^2)(\beta-\alpha)\frac{b_2}{c_3}\right\}.
\end{align*}
Inserting the expressions $b_2=b^2(\beta-\alpha-b\alpha)$ and $c_3=(1+b+b^2)(1+\alpha+b\alpha)$, we obtain
\begin{align*}
    \frac{(B_1+B_2)(1+\alpha+b\alpha)}{b(1+b)(1+b+b^2)\beta}=&b^2(1+\alpha+b\alpha)-(\beta-\alpha)b^2(\beta-\alpha-b\alpha)\\
    =&-b^2(1+\beta)(\beta-1-(2+b)\alpha).
\end{align*}
By assumption, we have $\beta-1-(2+b)\alpha\geq 0$, and hence $B_1+B_2\leq 0$. Hence, we derive $\Delta_{12}(u,v)\geq 0$ by \eqref{eq:deltaCase2}.

\textbf{Case $(iii)$: $d(u,v)=2$ and $\beta< 1+(2+b)\alpha$.} We have in this case that 
\begin{align*}
    \frac{n\Delta_{13}(u,v)}{C_1C_3(\lambda_3-\lambda_1)}
    =&d\lambda_2(\lambda_2-d-1)-\left((n-1)\lambda_2-nd\right)c_2\notag\\
    =&\left(d\lambda_2-(n-1)c_2\right)(\lambda_2-d-1)-(n-1-d)c_2\notag\\
  =&(1+b)\left\{d(1+\alpha+b\beta)-(n-1)(1+\alpha)\right\}(b-\beta+b\alpha+\alpha)\notag\\
  &\hspace{5cm}-(n-1-d)(1+b)(1+\alpha),  
\end{align*}
where we have used $\lambda_2=(1+b)(1+\alpha+b\beta)$, $d=(1+b+b^2)\beta$ and $c_2=(1+b)(1+\alpha)$. Applying \eqref{n}, we derive
\begin{align*}
    &\frac{n\Delta_{13}(u,v)}{C_1C_3(\lambda_3-\lambda_1)(1+b)}\\
    =&\left(db\alpha-db(\beta-\alpha)\frac{b_2}{c_3}\right)(b-\beta+b\alpha+\alpha)-\left(db(\beta-\alpha)+db(\beta-\alpha)\frac{b_2}{c_3}\right)\\
    =&db\left\{\alpha(b-\beta+b\alpha+\alpha)-(\beta-\alpha)\right\}-db(\beta-\alpha)\frac{b_2}{c_3}(b-\beta+b\alpha+\alpha+1)\\
    =&-db(1+\alpha)(\beta-\alpha-b\alpha) -\beta b^3(\beta-\alpha)\frac{\beta-\alpha-b\alpha}{1+\alpha+b\alpha}(b-\beta+b\alpha+\alpha+1),
\end{align*}
where we have inserted $b_2=b^2(\beta-\alpha-b\alpha)$, $c_3=(1+b+b^2)(1+\alpha+b\alpha)$ and $d=(1+b+b^2)\beta$. Rearranging the identity yields 
\begin{align}\label{eq:deltaCase3}
    &\frac{n(1+\alpha+b\alpha)\Delta_{13}(u,v)}{C_1C_3(\lambda_3-\lambda_1)\beta b(1+b)(\beta-\alpha-b\alpha)}\notag\\
    =&-(1+b+b^2)(1+\alpha)(1+\alpha+b\alpha)+b^2(\beta-\alpha)(\beta-b-b\alpha-\alpha-1)=:g(\beta).
\end{align}
For the function $g(\beta)$, we observe that
\begin{align*}
    \max_{1+(1+b)\alpha\leq \beta<1+(2+b)\alpha}g(\beta)\leq \max\{g(1+(1+b)\alpha), g(1+(2+b)\alpha)\}.
\end{align*}
Since $-(1+b+b^2)(1+\alpha)(1+\alpha+b\alpha)<0$, and the value $1+(1+b)\alpha$ lies in between $\alpha$ and $b+b\alpha+\alpha+1$, we have $g(1+(1+b)\alpha)<0$. Moreover, we check
\begin{align*}
    g(1+(2+b)\alpha)=(1+\alpha+b\alpha)(-b^3-(1+b+b^2)-(1+b)(1+\alpha))<0.
\end{align*}
Here, we used $\alpha\geq 0$ from Lemma \ref{lem1}.
Therefore, we have $\max_{1+(1+b)\alpha\leq \beta<1+(2+b)\alpha}g(\beta)<0$, and hence $\Delta_{13}(u,v)>0$ by \eqref{eq:deltaCase3}.

\textbf{Case $(iv)$: $d(u,v)=3$.}
Recall that $d-\lambda_1=\theta_1=(1+b)(\beta-\alpha)-1$. We derive that
\begin{align*}
\frac{n\Delta_{23}(u,v)}{C_2C_3(\lambda_3-\lambda_2)}=d\lambda_1(\lambda_1-d-1)=-d\lambda_1(1+b)(\beta-\alpha)< 0.
\end{align*}
Indeed, we have for any graph with $n$ vertices and maximal degree $d_{\max}$, we have $\lambda_1\leq \frac{n}{n-1}d_{\max}<d_{\max}$ by a trace argument. Hence we have $\lambda_1-d<0$ for any $d$-regular graph. Therefore, we obtain $\Delta_{23}(u,v)>0$.
This completes the proof.
\end{proof}

\begin{proof}[Proof of Theorem \ref{thm1} in case $b\geq 1$]
    In case $b\geq 1$, Theorem \ref{thm1} follows from Theorem \ref{condition}, by verifing the conditions via Lemma \ref{lemma:1+2-3}, Lemma \ref{lemma:L2-L(2+3)}, Lemma \ref{+di} and Lemma \ref{+dij}.
\end{proof}

\subsection{Proof of Theorem \ref{thm1} in case \texorpdfstring{$b\le-2$}{b <= -2}}\label{sec6}

In this section , we will prove Theorem \ref{thm1} in the case of $b\le-2$.

\begin{lem}\label{-di}
    Let $G$ be a distance-regular graph with classical parameters $(3,b,\alpha,\beta)$, where $b\le-2$. Then we have for any two distinct vertices $u$ and $v$ that
    $$\Delta_1(u,v)\ge0, \Delta_2(u,v)\ge0, \text{and} \  \Delta_3(u,v)\ge0.$$
\end{lem}

\begin{proof}
The proof is similar to that of Lemma \ref{+di}. 
Recalling that $C_1>0$, $C_2<0$, $C_3>0$, and $c_2\geq 0$ and the identity \eqref{eq:drgdij}, it remains to show
    \begin{itemize}
   \item [(i)] $(\lambda_2-d-1)(\lambda_3-d-1)+d-a_1-1\geq 0$, and $(\lambda_2-d)(\lambda_3-d)+d-c_2\geq 0$.
    \item [(ii)] $(\lambda_1-d-1)(\lambda_3-d-1)+d-a_1-1\leq 0$, and $(\lambda_1-d)(\lambda_3-d)+d\leq 0$.
    \item [(iii)] $(\lambda_1-d-1)(\lambda_2-d-1)+d-a_1-1\geq 0$, and $(\lambda_1-d)(\lambda_2-d)+d-c_2\geq 0$.
    \end{itemize}
By \eqref{eq:eigen_order}, we have in case $b\leq -2$ that
$\lambda_1=\gamma_2, \lambda_2=\gamma_3$ and $\lambda_3=\gamma_1$. This tells in particular that \[d-\lambda_1=\theta_2,\,\, d-\lambda_2=\theta_3, \,\,\text{and}\,\,d-\lambda_3=\theta_1.\] Using the calculations in the proof of Lemma \ref{+di}, we check $(i)-(iii)$ as follows. 

  \textbf{Case $(i)$}: By \eqref{eq:(1+theta1)(1+theta3)}, we have
\begin{align*}
    &(\lambda_2-d-1)(\lambda_3-d-1)+(d-a_1-1)=(1+\theta_3)(1+\theta_1)+b_1\\
    =&-b(b+b^2)(\beta-\alpha)>0.
\end{align*}
In the last inequality, we have used the fact that $\alpha<-1$ due to Lemma \ref{lem1}. Furthermore, we derive by \eqref{eq:theta1theta3} that
\begin{align*}
    &(\lambda_2-d)(\lambda_3-d)+d-c_2=\theta_3\theta_1+d-c_2\\
    =&(1+b+b^2)(1+(1+b)\alpha-b\beta)-(1+b)(1+\alpha)\\
    =&b\left\{b+(1+b)^2\alpha-(1+b+b^2)\beta\right\}>0.
\end{align*}

\textbf{Case $(ii)$}: We derive by \eqref{eq:(1+theta1)(1+theta2)} that
\begin{align*}
    &(\lambda_1-d-1)(\lambda_3-d-1)+d-a_1-1=(1+\theta_2)(1+\theta_1)+b_1\\
    =&(1+b)(\beta-\alpha)(\beta-b\alpha-\alpha)<0.
\end{align*}
Furthermore, we deduce from \eqref{eq:theta1theta2} that
\begin{align*}
    &(\lambda_1-d)(\lambda_3-d)+d=\theta_2\theta_1+d\\
    =&(1+b)(\beta-\alpha-1)(\beta-b\alpha-\alpha)+c_2\\
    =&(1+b)\left\{\beta(\beta-b\alpha-\alpha)-(1+\alpha)(\beta-b\alpha-\alpha-1)\right\}<0.
\end{align*}

\textbf{Case $(iii)$}: We deduce from \eqref{eq:(1+theta2)(1+theta3)} and \eqref{eq:theta2theta3} that
\begin{align*}
    &(\lambda_1-d-1)(\lambda_2-d-1)+d-a_1-1=(1+\theta_2)(1+\theta_3)+b_1\\
    =&b^2(1+b)(1+\alpha)> 0,
\end{align*}
and
\begin{align*}
    &(\lambda_1-d)(\lambda_2-d)+d-c_2=\theta_2\theta_3+d-c_2=c_2(b+b^2)>0.
\end{align*}This completes the proof.
\end{proof}

    \begin{lem}\label{-dij}
Let $G$ be a distance-regular graph with classical parameters  $(3,b,\alpha,\beta)$, where $b\le-2$. For any two distinct vertices $u$ and $v$, we have the following estimates. 
\begin{itemize}
    \item [(i)] If $d(u,v)=1$, then $\Delta_{12}(u,v)\ge0$,
\item [(ii)]  If $d(u,v)=2$, then $\Delta_{12}(u,v)\ge0$,
    \item [(iii)] If $d(u,v)=3$, then $\Delta_{23}(u,v)\ge0$.
\end{itemize}
\end{lem}

\begin{proof}
    As in the proof of Lemma \ref{+dij}, we have for $\{i,j,k\}=\{1,2,3\}$ that 
    \begin{align*}
\frac{n\Delta_{ij}(u,v)}{C_iC_j(\lambda_j-\lambda_i)}=
\begin{cases}
d\lambda_k(\lambda_k-d-1), &\text{if $d(u,v)=3$;}\\
d\lambda_k(\lambda_k-d-1)-\left((n-1)\lambda_k-nd\right)c_2, &\text{if $d(u,v)=2$;}\\
-(n-1-d)\lambda_k(\lambda_k-d-1)+\left((n-1)\lambda_k-nd\right)b_1, &\text{if $d(u,v)=1$.}
\end{cases}
\end{align*}

Next, we check the three cases in reverse order. 

\textbf{Case $(iii)$: $d(u,v)=3$.} We have
\begin{align*}
    \frac{n\Delta_{23}(u,v)}{C_2C_3(\lambda_3-\lambda_2)}=d\lambda_1(\lambda_1-d-1)<0.
\end{align*}
The inequality follows from the general fact that $\lambda_1<d$ for any $d$-regular graph, which can be shown by a trace argument. The inequality can also be checked by inserting the expressions of $\lambda_1=\gamma_2$ and $d$. Therefore, we have $\Delta_{23}(u,v)>0$.

\textbf{Case $(ii)$: $d(u,v)=2$.} In this case we have
\begin{align*}
    \frac{n\Delta_{12}(u,v)}{C_1C_2(\lambda_2-\lambda_1)}=&d\lambda_3(\lambda_3-d-1)-\left((n-1)\lambda_3-nd\right)c_2\\
    &=\left(d\lambda_3-(n-1)c_2\right)(\lambda_3-d-1)-(n-1-d)c_2.
\end{align*}
Recall that $c_2=(1+b)(1+\alpha)$, $\lambda_3=\gamma_1=1+\alpha+b\alpha+b^2\beta$, and $d=(1+b+b^2)\beta$. We derive by \eqref{n} that
\begin{align*}
    &\frac{n\Delta_{12}(u,v)}{C_1C_2(\lambda_2-\lambda_1)}\\
    =&\left(d\lambda_3-(n-1)(1+\alpha)(1+b)\right)(\alpha-\beta)(1+b)-(n-1-d)(1+\alpha)(1+b)\\
    =&\left(d\lambda_3-d(1+\alpha)(1+b)-db(\beta-\alpha)\left(1+\frac{b_2}{c_3}\right)(1+b)\right)(\alpha-\beta)(1+b)\\
    &\hspace{7cm}-db(\beta-\alpha)\left(1+\frac{b_2}{c_3}\right)(1+b)\\
    =&d(b^2\beta-b)(\alpha-\beta)(1+b)-db(\beta-\alpha)\left(1+\frac{b_2}{c_3}\right)(1+b)((\alpha-\beta)(1+b)+1).
\end{align*}
Recall that 
\begin{align*}
   d\left(1+\frac{b_2}{c_3}\right)=\beta\frac{(1+\alpha)(1+b+b^2)+b(\alpha+b\beta)}{1+\alpha+b\alpha}.
\end{align*}
Then, we continue to calculate 
\begin{align}
    &\frac{n\Delta_{12}(u,v)}{C_1C_2(\lambda_2-\lambda_1)}=\frac{b\beta(\alpha-\beta)(1+b)}{1+\alpha+b\alpha}\big\{(1+b+b^2)(b\beta-1)(1+\alpha+b\alpha)\notag\\
    &\hspace{3.5cm}+\left((1+\alpha)(1+b+b^2)+b(\alpha+b\beta)\right)((\alpha-\beta)(1+b)+1)\big\}.\label{eq:delta12Case2}
\end{align}
By the rearrangement $(\alpha-\beta)(1+b)-1=\alpha+b\alpha-\beta-(b\beta-1)$, we deduce
\begin{align*}
    &(1+b+b^2)(b\beta-1)(1+\alpha+b\alpha)+\left((1+\alpha)(1+b+b^2)+b(\alpha+b\beta)\right)((\alpha-\beta)(1+b)+1)\\
    =&b^2\left(\alpha+b\alpha-\beta\right)(b\beta-1)+\left\{(1+\alpha)(1+b+b^2)+b(\alpha+b\beta)\right\}(\alpha+b\alpha-\beta)\\
    =&(\alpha+b\alpha-\beta)(1+b)(1+\alpha+b\alpha+b^2\beta).
\end{align*}
Inserting into \eqref{eq:delta12Case2} yields
\begin{align}
    \frac{n\Delta_{12}(u,v)}{C_1C_2(\lambda_2-\lambda_1)}=&d\lambda_3(\lambda_3-d-1)-\left((n-1)\lambda_3-nd\right)c_2\notag\\
    =&\frac{b\beta(\alpha-\beta)(1+b)}{1+\alpha+b\alpha}(\alpha+b\alpha-\beta)(1+b)(1+\alpha+b\alpha+b^2\beta).\label{eq:delta12Case2final}
\end{align}
This implies that $\Delta_{12}(u,v)>0$.

\textbf{Case $(i)$: $d(u,v)=1$.} We have in this case that
\begin{align*}
   \frac{-c_2}{b_1}\cdot \frac{n\Delta_{12}(u,v)}{C_1C_2(\lambda_2-\lambda_1)}=&(n-1-d)\lambda_3(\lambda_3-d-1)\frac{1+\alpha}{b(\beta-\alpha)}-\left((n-1)\lambda_3-nd\right)c_2\\
   =&d\left(1+\frac{b_2}{c_3}\right)\lambda_3(\lambda_3-d-1)-\left((n-1)\lambda_3-nd\right)c_2.
\end{align*}
In the above, we have used $b_1=(b+b^2)(\beta-\alpha)$, $c_2=(1+b)(1+\alpha)$ and \eqref{n}. Using \eqref{eq:delta12Case2final}, we obtain
\begin{align*}
   &\frac{-c_2}{b_1}\cdot \frac{n\Delta_{12}(u,v)}{C_1C_2(\lambda_2-\lambda_1)}\\
   =&\frac{b\beta(\alpha-\beta)(1+b)}{1+\alpha+b\alpha}(\alpha+b\alpha-\beta)(1+b)(1+\alpha+b\alpha+b^2\beta)+d\lambda_3(\lambda_3-d-1)\frac{b_2}{c_3}.
\end{align*}
Observing that 
\begin{align*}
    d\lambda_3(\lambda_3-d-1)\frac{b_2}{c_3}=\frac{b\beta(\alpha-\beta)(1+b)}{1+\alpha+b\alpha}b(\beta-\alpha-b\alpha)(1+\alpha+b\alpha+b^2\beta)
\end{align*}
Hence, we derive
\begin{align*}
   \frac{-c_2}{b_1}\cdot \frac{n\Delta_{12}(u,v)}{C_1C_2(\lambda_2-\lambda_1)}=&\frac{b\beta(\alpha-\beta)(1+b)}{1+\alpha+b\alpha}(\alpha+b\alpha-\beta)(1+\alpha+b\alpha+b^2\beta)>0.
\end{align*}
Therefore, we have $\Delta_{12}(u,v)>0$. This completes the proof. 
\end{proof}

\begin{proof}[Proof of Theorem \ref{thm1} in case $b\leq -2$]
   Combining Lemma \ref{lemma:1+2-3}, Lemma \ref{lemma:L2-L(2+3)}, Lemma \ref{-di}, Lemma \ref{-dij} and Theorem \ref{condition}, we prove Theorem \ref{thm1} in case $b\leq -2$. 
\end{proof}

\begin{rmk}
Known distance regular graphs with classical parameters of diameter at least $3$ are listed in \cite[Tables 6.1 and 6.2]{BCN89}. For classification results of distance regular graphs with classical parameters of diameter at least $4$ and $b\leq -2$ or $b=1$, we refer to \cite{BCN89,DKT16,TLHHG24,Weng99} and the references therein. It is natural to expect the monotonic normalized heat diffusion property holds for any distance regular graphs with classical parameters. 
Recall that Proposition \ref{prop:eig} provides explicit expressions for all $D+1$ distinct Laplacian eigenvalues for any distance-regular graph with classical parameters $(D, b, \alpha, \beta)$. With this proposition in hand, the approach developed herein may pave the way for confirming the MNHD property for the cases where the number of distinct eigenvalues is greater than $4$. 
\end{rmk}

\section{Further Discussion} \label{sec7}

The distance-regular graphs with diameter $3$ are classified as antipodal, bipartite, and primitive \cite{Biggs82}. A distance-regular graph with an intersection array \(\{d, b_1, b_2; 1, c_2, c_3\}\) is antipodal if $b_2=1$ and $c_3=d$. In fact, the general form of the intersection array for an antipodal distance-regular graph of diameter $3$ is \(\{b_0, b_1, b_2; c_1, c_2, c_3\}=\{d, m\gamma, 1; 1, \gamma, d\}\), where $m$ is a positive integer \cite[Theorem 2]{Biggs82}. A distance-regular graph of diameter $3$ is primitive if it is neither antipodal nor bipartite. 

For the bipartite case, the MNHD property is true by the result of Kubo and Namba \cite{KN22}. For the antipodal case, we show that the MNHD property holds still.
\begin{thm}
    The antipodal distance-regular graphs with diameter $3$ have the MNHD property.
\end{thm}
\begin{proof}
Let $G$ be an antipodal distance-regular graph with intersection array $\{d,m\gamma,1;1,\gamma,d\}$, where $m$ is a positive integer. Let us denote $M:=\sqrt{4d + (d  - \gamma - m\gamma-1)^2}$. The non-trivial Laplacian eigenvalues are listed below,
\begin{align}
\notag    \lambda_1&=d + \frac{1}{2} (1 - d + \gamma + m\gamma - M),\\
\notag \lambda_2&=1 + d,\\
 \lambda_3&=d + \frac{1}{2} (1 - d + \gamma + m\gamma + M),\notag
\end{align}
with $\lambda_1<\lambda_2<\lambda_3$.
The adjacency eigenvalues of an antipodal distance-regular graph of diameter $3$ have been calculated in \cite[Page 75]{Biggs82}. One can also derive them directly from the intersection matrix as explained in \cite[Page 13]{DKT16}.

We further check that $n=1+d+m+dm$. 

Next, we check the conditions in Theorem \ref{condition}. First, we observe that $\lambda_1+\lambda_2-\lambda_3\geq0$.
Indeed, we have $\lambda_1+\lambda_2-\lambda_3=1 + d - M$. As $(1+d)^2-M^2=\gamma(1+m)(2d-m-m\gamma-2)\geq0$, we derive $1+d-M\geq 0$.

Moreover, we have $L^2(u,v) - L(u,v)(\lambda_2 + \lambda_3) \geq 0$ for any two adjacent vertices $u$ and $v$.
Recall that $L(u,v)=-2d+a_1=-2d+(d-1-m\gamma)$ and $L(u,v)=-1$. We get $L^2(u,v) - L(u,v)(\lambda_2 + \lambda_3) = \frac{1}{2} (1 + d + \gamma - m\gamma + M)\geq 0$, as desired.

 \textbf{Case $(i)$: $d(u,v)=1$.} It follows from the identity \eqref{eq:drgdij} that 
 \begin{align*}
     \Delta_1(u,v)/C_1 =(\lambda_2-d-1)(\lambda_3-d-1)+d-a_1-1.
 \end{align*}
Noting that $\lambda_2=1 + d$ and $d-a_1-1=b_1=m\gamma$, we deduce $\Delta_1/C_1=m\gamma \geq 0$. Since $C_1>0$, we obtain $\Delta_1(u,v)\geq 0$.

Similarly, we have 
\begin{align*}
   \Delta_3(u,v)/C_3 =&(\lambda_1-d-1)(\lambda_2-d-1)+d-a_1-1=m\gamma>0.
\end{align*}
This implies $\Delta_3(u,v)\geq 0$.
For $\Delta_2(u,v)$, we calclulate
\begin{align*}
   \Delta_2(u,v)/C_2 =&(\lambda_1-d-1)(\lambda_3-d-1)+d-a_1-1\\
   =&\frac{1}{2} (-1 - d + \gamma + m\gamma - M)\cdot\frac{1}{2} (-1 - d + \gamma + m\gamma + M)+m\gamma\\
   =&-\gamma<0.
\end{align*} 
Since $C_2<0$, we obtain $\Delta_2(u,v)>0$.

By \eqref{eq:Deltaij_simple}, we have 
\begin{align*}
   \frac{n\Delta_{12}(u,v)}{C_1C_2(\lambda_2-\lambda_1)} =&-(n-1-d)\lambda_3(\lambda_3-d-1)+\left((n-1)\lambda_3-nd\right)b_1\\
   =&-(n-1-d)\lambda_3(\lambda_3-d-1)+n(\lambda_3-d)b_1-\lambda_3b_1\\
   =&-\lambda_3b_1<0,
\end{align*}
where we used 
\begin{align*}
    &-(n-1-d)\lambda_3(\lambda_3-d-1)+n(\lambda_3-d)b_1\\
    =&-m(d+1)\frac{1}{4}(1+d+\gamma+m\gamma+M)(-1-d+\gamma+m\gamma+M)\\
    &\hspace{3cm}+\frac{m\gamma}{2}(d+1)(m+1)(1-d+\gamma+m\gamma+M)\\
    =&-\frac{m(d+1)}{4}\left\{(M+\gamma(1+m))^2-(d+1)^2-2\gamma(m+1)(1-d+\gamma(m+1)+M)\right\}\\
    =&0.
\end{align*}
Now we apply Theorem \ref{condition} to conclude that the function $r_t(u,v):=\frac{H_t(u,v)}{H_t(u,u)}$ is monotonically non-decreasing in $t$ for any two adjacent vertices $u$ and $v$. 

\textbf{Case $(ii)$: $d(u,v)=2$.} Similar as in Case (i), we have \begin{align*}
    \Delta_1(u,v)/C_1=&(\lambda_2-d)(\lambda_3-d)+d-c_2\\
    =&\frac{1}{2} (1 - d + \gamma + m\gamma + M)+d-\gamma\\
    =&\frac{1}{2} \left( 1 + d - \gamma + m \gamma + M \right)>0,
\end{align*}
\begin{align*}
    \Delta_2(u,v)/C_2=(\lambda_1-d)(\lambda_3-d)+d-c_2=-\gamma<0.
\end{align*}
and
\begin{align*}
    \Delta_3(u,v)/C_3=(\lambda_1-d)(\lambda_2-d)+d-c_2=\frac{1}{2} \left( 1 + d - \gamma + m \gamma - M \right)>0,
\end{align*}
where in the last inequlaity we used $(1+d-\gamma+m\gamma)^2-M^2=-4\gamma(1-dm+m\gamma)>0$. Therefore, we obtain $\Delta_i(u,v)\geq 0$ for each $i=1,2,3$.

Furthermore, we check
\begin{align*}
   \frac{n\Delta_{13}(u,v)}{C_1C_3(\lambda_3-\lambda_1)} = d\lambda_2(\lambda_2-d-1)-\left((n-1)\lambda_2-nd\right)c_2=-(1+d)m\gamma<0,
\end{align*}
and, hence, $\Delta_{13}(u,v)<0$. We apply Theorem \ref{condition} to conclude that the function $r_t(u,v)$ is monotonically non-decreasing in $t$ for any two  vertices $u$ and $v$ with $d(u,v)=2.$ 

\textbf{Case $(iii)$: $d(u,v)=3$.} Similar as in Case (i), we have \begin{align*}
    \Delta_1(u,v)/C_1=(\lambda_2-d)(\lambda_3-d)+d=\frac{1}{2} \left( 1 + d + \gamma + m \gamma + M \right)>0,
\end{align*}
$$\Delta_2(u,v)/C_2=(\lambda_1-d)(\lambda_3-d)+d=0,$$
and
$$\Delta_3(u,v)/C_3=(\lambda_1-d)(\lambda_2-d)+d=\frac{1}{2} \left( 1 + d + \gamma + m \gamma - M \right)>0,$$
where in the last inequality we used $(1+d+\gamma+m\gamma)^2-M^2=4d(1+m)\gamma>0$. Hence, we obtain $\Delta_i(u,v)\geq 0$ for each $i=1,2,3$.

Furthermore, we check
\begin{align*}
   \frac{n\Delta_{13}(u,v)}{C_1C_3(\lambda_3-\lambda_1)} =d\lambda_2(\lambda_2-d-1)=0.
\end{align*}
Applying Theorem \ref{condition}, we derive that the function $r_t(u,v)$ is monotonically non-decreasing in $t$ for any two vertices $u$ and $v$ with $d(u,v)=3$. 
\end{proof}

Our main Theorem \ref{thm1} implies that the MNHD property holds for primitive distance-regular graphs with classical parameters of diameter $3$. It is natural to ask whether all distance-regular graphs with diameter $3$ have the MNHD property or not.


\section*{Acknowledgement}
We are very grateful to the anonymous referee for suggestions that improved the readability of our manuscript. H.Z. is very grateful to Kaizhe Chen and Chenhui Lv for
their helpful discussions. This work is supported by National Key R \& D Program of China 2023YFA1010200 and National Natural Science Foundation of China No. 12031017 and No. 12431004. 

\subsection*{Availability of Data}  No data was used for the research described in the article.

\subsection*{Declarations of Conflict of Interest}
The authors do not have any possible conflicts of
interest.

\end{document}